\documentclass[12pt]{article}
\usepackage{amsfonts,epsf,amsmath,amssymb,graphicx}
\usepackage{epstopdf}
\usepackage[top=2.5cm,bottom=2.5cm,left=2.5cm,right=2.5cm]{geometry}
\newtheorem{theorem}{\bf Theorem}[section]

\newtheorem{proposition}[theorem]{\bf Proposition}

\newcommand{\qed}{\hfill $\square$ \bigskip}

\def\big{\bigskip }

\begin{document}
\baselineskip=0.28in

\vspace*{40mm}

\begin{center}
{\LARGE \bf  The edge-Hosoya polynomial of benzenoid chains}
\bigskip \bigskip

{\large \sc Niko Tratnik$^{a}$, Petra \v Zigert Pleter\v sek$^{a,b}$}

\smallskip
{\em  $^a$Faculty of Natural Sciences and Mathematics, University of Maribor,  Slovenia} \\
{\em $^b$Faculty of Chemistry and Chemical Engineering, University of Maribor, Slovenia} \\

e-mail: {\tt niko.tratnik@um.si, \tt petra.zigert@um.si}

\bigskip\medskip

  (Received \today) 

\end{center}

\vspace{3mm}\noindent

\begin{abstract}
The Hosoya polynomial is a well known vertex-distance based polynomial, closely correlated to the Wiener index and the hyper-Wiener index, which are widely used molecular-structure descriptors. In the present paper we consider the edge version of the Hosoya polynomial. For a connected graph $G$ let $d_e(G,k)$ be the number of (unordered) 
edge pairs at distance $k$\,. Then the edge-Hosoya polynomial of $G$ 
is $H_e(G,x) = \sum_{k \geq 0} d(G,k)\,x^k$\,. We investigate the edge-Hosoya polynomial of important chemical graphs known as benzenoid chains and derive the recurrence relations for them. These recurrences are then solved for linear benzenoid chains, which are also called polyacenes. 

 \end{abstract}

\baselineskip=0.30in

\big
\section{Introduction}
In 1988 H. Hosoya introduced some counting polynomials in chemistry \cite{hosoya}, among them the Wiener polynomial, nowadays known as the {\it Hosoya polynomial}. Since then the topic has been intensively researched - for example, the Hosoya polynomial of benzenoid chains \cite{gut,xu-zh}, carbon nanotubes \cite{zig-zag}, circumcoronene series \cite{circum}, Finonacci and Lucas cubes \cite{klav} has been calculated. For some recent results on the Hosoya polynomial see \cite{deu}. What makes the Hosoya polynomial interesting in chemistry is especially the strong connection to the Wiener index and the hyper-Wiener index \cite{cash}.
\bigskip

\noindent
The Hosoya polynomial is based on the distances between pairs of vertices in a graph, and similar concept has been introduced in \cite{behm} for distances between pairs of edges under the name the edge-Hosoya polynomial. The authors defined the distance between two edges $e=ab$ and $f =xy$ of a graph $G$ as $\widehat{d}(e,f) = \min \lbrace d(a,x), d(a,y), d(b,x), d(b,y) \rbrace $. However, the distance between two edges of graph $G$ can also be defined as the distance between vertices $e$ and $f$ in the line graph $L(G)$. In \cite{iranmanesh-2009} it was discussed that the pair $(E(G), \widehat{d})$ is not a metric space and therefore it was suggested that the distance between pairs of edges should be considered in the line graph. Hence, in the present paper we use the last definition of distance between edges for the edge-Hosoya polynomial $H_e(G,x)$. The other version of the edge-Hosoya polynomial using distance $\widehat{d}$ is denoted by $\widehat{H}_e(G,x)$. In \cite{behm} the relation between both versions of the edge-Hosoya polynomial was established, but it is not completely correct. In the present paper the true relation is proved in Proposition \ref{popravek}. Moreover, the relation between the Hosoya polynomial and the edge-Hosoya polynomial for trees was established in \cite{tr-zi}.
\bigskip

\noindent
Next, we describe the strong connections of the edge-Hosoya polynomial to the edge-Wiener index and the edge-hyper-Wiener index, which show the importance of the edge-Hosoya polynomial in chemistry. The \textit{edge-Wiener index} \cite{iranmanesh-2009} and the \textit{edge-hyper-Wiener index} \cite{edge-hyper} of a connected graph $G$ are defined as
$$W_e(G) = \frac{1}{2} \sum_{e \in E(G)} \sum_{f \in E(G)} d(e,f),$$
$$WW_e(G) = \frac{1}{4} \sum_{e \in E(G)} \sum_{f \in E(G)} d(e,f) + \frac{1}{4} \sum_{e \in E(G)} \sum_{f \in E(G)} d(e,f)^2$$
and have been much investigated in recent years (for example, see \cite{chen,cre-trat1,kelenc,knor,tratnik}). The equality 
\begin{equation}\label{ew}
W_e(G) = H_e'(G,1)
\end{equation}
between the edge-Wiener index and the edge-Hosoya polynomial is obvious and the equality between the edge-hyper-Wiener index and the edge-Hosoya polynomial was shown in \cite{tr-zi}:

\begin{equation} \label{ehw}
WW_e(G) = H_e'(G,1) + \frac{1}{2}H_e''(G,1).
\end{equation}

\noindent
The main aim of the paper is deriving the recursive relations of the edge-Hosoya polynomial for benzenoid chains. This continues the research from \cite{gut}, where similar methods were used for determining the Hosoya polynomial. However, in our case some additional recursions are needed. In the next section we introduce basic concepts and the relation between $H_e(G,x)$ and $\widehat{H}_e(G,x)$ is established. In sections \ref{anne} and \ref{chains} the recursive relations for the edge-Hosoya polynomial of benzenoid chains are calculated and solved for polyacenes in the final section. In a similar way, the edge-Hosoya polynomial can be computed for any benzenoid chain.

\section{Preliminaries}

\textit{Distance} $d(x,y)$ between vertices $x,y \in V(G)$ is defined as the usual shortest-path distance. If $G$ is a connected graph and if $d_v(G,k)$ 
is the number of (unordered) pairs of its vertices that are at distance $k$\,, 
then the \textit{Hosoya polynomial} of $G$ is defined as
$$
H(G,x) = \sum_{k \geq 0} d_v(G,k)\,x^k.
$$

The \textit{distance} between edges $e$ and $f$ of a graph $G$, $d(e,f)$, is defined as the distance between vertices $e$ and $f$ in the line graph $L(G)$. The distance between a vertex $w$ and an edge $e=uv$ is defined as $d(w,e) = \min \{d(w,u), d(w,v) \}$.

\noindent
If $G$ is a connected graph with $m$ edges, and if $d_e(G,k)$ 
is the number of (unordered) pairs of its edges that are at distance $k$\,, 
then the \textit{edge-Hosoya polynomial} of $G$ is 
$$
H_e(G,x) = \sum_{k \geq 0} d_e(G,k)\,x^k.
$$
For convenience we write $d(G,k)$ for $d_e(G,k)$ and set $d(G,k)=0$ for $k < 0$. Note that $d(G,0) = m$.  Also, the following proposition is obvious.
\begin{proposition}
Let $G$ be a connected graph. Then
$$H_e(G,x) = H(L(G),x).$$
\end{proposition}

On the other hand, for edges $e = ab$ and $f = xy$ of a graph $G$ it is also legitimate to set
\begin{equation*}
\label{eq:hat-d}
\widehat{d}(e,f) = \min \lbrace d(a,x), d(a,y), d(b,x), d(b,y) \rbrace\,.
\end{equation*}
Replacing $d$ with $\widehat{d}$, a variant of the edge-Hosoya polynomial from~\cite{behm} is obtained, let us denote it with $\widehat{H}_e(G,x)$. The following proposition gives us the relation between both variants of the edge-Hosoya polynomial. Similar result was also obtained in \cite{behm} as Theorem $1$, but there it is not completely correct.

\begin{proposition}
\label{popravek}
Let $G$ be a connected graph. Then 
$$H_e(G,x) = x\left(\widehat{H}_e(G,x) - |E(G)|\right) + |E(G)|.$$
\end{proposition}

\begin{proof}
If $k= 0$, then it is obvious that $d(G,0) = |E(G)|$. Let $k=1$. The number $d(G,1)$ is equal to the number of pairs of edges $e$ and $f$ such that $d(e,f)=1$. However, it follows that $d(G,1)$ is the number of pairs of edges $e$ and $f$ for which $\widehat{d}(e,f)=0$, reduced by the number of edges of $G$, $|E(G)|$. Therefore, the coefficient before $x^1$ is the same on the both sides.
Finally, let $k>1$. Then $d(G,k)$ is equal to the number of pairs of edges $e$ and $f$ for which $d(e,f) = \widehat{d}(e,f) + 1 =k $, therefore, the coefficient before $x^k$ in $H_e(G,x)$ is equal to the coefficient before $x^{k-1}$ in $\widehat{H}_e(G,x)$. The proof is complete. \qed
\end{proof}
\smallskip


Finally, we need some additional definitions. For a graph $G$\,, $k\geq 0$\,, and a vertex $v \in V(G)$\,, let 
$d(G,v,k)$ be the number of edges of $G$ at distance $k$ from 
$v$\,. This time, $d(G,v,0)=deg(v)$\,, and for $k<0$ we set 
$d(G,v,k)=0$\,. We now define $H_e(G,v,x)$ as 
$$
H_e(G,v,x)=\sum_{k\geq 0} d(G,v,k) x^k \ .
$$

\noindent
For a graph $G$\,, $k\geq 0$\,, and an edge $e \in E(G)$\,, let 
$d(G,e,k)$ be the number of edges of $G$ at distance $k$ from 
$e$\,. This time, $d(G,e,0)=1$\,, and for $k<0$ we set 
$d(G,e,k)=0$\,. We now define $H_e(G,e,x)$ as 
$$
H_e(G,e,x)=\sum_{k\geq 0} d(G,e,k) x^k \ .
$$

\section{Annelating a $6$-cycle}
\label{anne}

In this section we describe how the edge-Hosoya polynomial of a graph $G_0$ can be used to compute the edge-Hosoya polynomial of a graph $G$, obtained by attaching a $6$-cycle to $G_0$. In chemistry, such an operation is known as \textit{annelation}. The obtained result will be used in the next section, where we consider benzenoid chains. 

\begin{figure}[!htb]
	\centering
		\includegraphics[scale=0.7, trim=0cm 0cm 1cm 0cm]{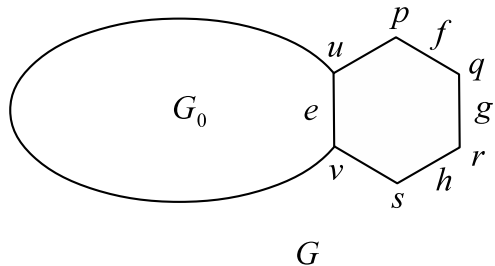}
\caption{A graph $G$ obtained by annelating a $6$-cycle to $G_0$ over an edge $e=uv$.}
	\label{dod}
\end{figure}

\begin{theorem}
\label{rekurzija}
Let the graph $G$ be obtained by annelating a $6$-cycle to the graph $G_0$ over an edge $e=uv$. Then
\begin{eqnarray*}
H_e(G,x) & = & H_e(G_0,x) + (x+x^2)H_e(G_0,u,x) + (x+x^2)H_e(G_0,v,x) \\
 & + & x^2H_e(G_0,e,x) + 5 + 4x + 3x^2 + 3x^3.
\end{eqnarray*}

\end{theorem}

\begin{proof}
Let $e$, $u$ and $v$ be as in the Figure \ref{dod} and $k$ an integer. We define $d(G_0,k,e,u,v)$ as the sum
$$d(G_0,k) + d(G_0,u,k-1) + d(G_0,u,k-2) + d(G_0,v,k-1) + d(G_0,v,k-2) + d(G_0,e,k-2).$$

\noindent
Obviously, it follows from Figure \ref{dod}, that for every $k\geq 0$ it holds
$$d(G,k) = d(G_0,k,e,u,v) + \left\{ \begin{array}{ll} 
0 \ ; & k > 3 \\ 
3 \ ; & k = 3 \\ 
3 \ ; & k = 2 \\
4 \ ; & k = 1 \\
5 \ ; & k = 0. \end{array} \right.$$

\noindent
Therefore, 
$$H_e(G,x) = \sum_{k \geq 0}d(G,k) = \sum_{k \geq 0}d(G_0,k,e,u,v) + 5 + 4x + 3x^2 + 3x^3,$$
which implies
\begin{eqnarray*}
H_e(G,x) & = & H_e(G_0,x) + xH_e(G_0,u,x) + x^2H_e(G_0,u,x) \\
 & + & xH_e(G_0,v,x) + x^2H_e(G_0,v,x) + x^2H_e(G_0,e,x)\\
& + & 5 + 4x + 3x^2 + 3x^3 
\end{eqnarray*}
and the proof is complete. \qed
\end{proof}

\noindent
In the following theorem we describe how the edge-Hosoya polynomials with a fixed vertex or an edge can be obtained recursively.
\begin{theorem}
\label{rek1}
Let $G$, $G_0$, $u$, $v$, $p$, $q$, $r$, $s$, $e$, $f$ and $g$ be as in Figure \ref{dod}. Then
\begin{enumerate}
\item $H_e(G,p,x) = xH_e(G_0,u,x) + 2 + x + 2x^2$,
\item $H_e(G,q,x) = x^2H_e(G_0,u,x) + 2 + 2x + x^2$,
\item $H_e(G,r,x) = x^2H_e(G_0,v,x) + 2 + 2x + x^2$,
\item $H_e(G,s,x) = xH_e(G_0,v,x) + 2 + x + 2x^2$,
\item $H_e(G,f,x) = x^2H_e(G_0,u,x) + 1 + 2x + x^2 + x^3$,
\item $H_e(G,g,x) = x^2H_e(G_0,e,x) + 1 + 2x + x^2 + x^3$,
\item $H_e(G,h,x) = x^2H_e(G_0,v,x) + 1 + 2x + x^2 + x^3$.
\end{enumerate}
\end{theorem}

\begin{proof}
It suffices to prove cases $1.$, $2.$, $5.$ and $6.$, since the others are similar.
\begin{enumerate}
\item Obviously, it follows from Figure \ref{dod} that
$$d(G,p,k) = d(G_0,u,k-1) + \left\{ \begin{array}{ll} 
0 \ ; & k > 2 \\ 
2 \ ; & k = 2 \\ 
1 \ ; & k = 1 \\
2 \ ; & k = 0. \end{array} \right.$$

Hence,
$$H_e(G,p,x) = xH_e(G_0,u,x) + 2 + x + 2x^2.$$

\item It follows from Figure \ref{dod} that
$$d(G,q,k) = d(G_0,u,k-2) + \left\{ \begin{array}{ll} 
0 \ ; & k > 2 \\ 
1 \ ; & k = 2 \\ 
2 \ ; & k = 1 \\
2 \ ; & k = 0. \end{array} \right.$$

Therefore,
$$H_e(G,q,x) = x^2H_e(G_0,u,x) + 2 + 2x + x^2.$$

\item This case is similar to case $2.$

\item This case is similar to case $1.$

\item It follows from Figure \ref{dod} that
$$d(G,f,k) = d(G_0,u,k-2) + \left\{ \begin{array}{ll} 
0 \ ; & k > 3 \\
1 \ ; & k = 3 \\
1 \ ; & k = 2 \\ 
2 \ ; & k = 1 \\
1 \ ; & k = 0. \end{array} \right.$$

This implies
$$H_e(G,f,x) = x^2H_e(G_0,u,x) + 1 + 2x + x^2 + x^3.$$

\item It follows from Figure \ref{dod} that
$$d(G,g,k) = d(G_0,e,k-2) + \left\{ \begin{array}{ll} 
0 \ ; & k > 3 \\
1 \ ; & k = 3 \\
1 \ ; & k = 2 \\ 
2 \ ; & k = 1 \\
1 \ ; & k = 0. \end{array} \right.$$

So,
$$H_e(G,g,x) = x^2H_e(G_0,e,x) + 1 + 2x + x^2 + x^3.$$

\item This case is similar to case $5.$
\end{enumerate}
\qed
\end{proof}

\section{Application to benzenoid chains}
\label{chains}

A {\em benzenoid chain} with $h$ hexagons is a graph defined recursively as follows. 
If $h=1$ then $B'$ is the cycle on six vertices. For $h>1$ we obtain
$B$ from a benzenoid chain $B'$ with $h-1$ hexagons by attaching the $h$th
hexagon along an edge $e$ of the $(h-1)$st hexagon, where the end-vertices 
of $e$ are of degree 2 in the hexagonal chain $B'$. 

We now apply the previous recursive relations to benzenoid chains. 
Let $B_h$ be a benzenoid chain with $h$ hexagons obtained by 
adding a 6-cycle to $B_{h-1}$ over an edge $u_{h-1}v_{h-1}$\,. 
Then by Theorem \ref{rekurzija} we have
\begin{eqnarray*}
H_e(B_h,x) & = & H_e(B_{h-1},x)+(x+x^2)\Bigl[ 
H_e(B_{h-1},u_{h-1},x)+
             H_e(B_{h-1},v_{h-1},x)\Bigr] \\
         &  & + x^2H_e(B_{h-1},e_{h-1},x) +5+4x+3x^2+3x^3\,.
\end{eqnarray*}
Furthermore, let $u_hv_h$ be the edge that will be used in the 
subsequent annelation, that is, in the process $B_h \to 
B_{h+1}$\,. There are three possibilities for the edge
$u_hv_h$ and these are shown in Figure \ref{moznosti}.

\begin{figure}[!htb]
	\centering
		\includegraphics[scale=0.8, trim=0cm 0cm 1cm 0cm]{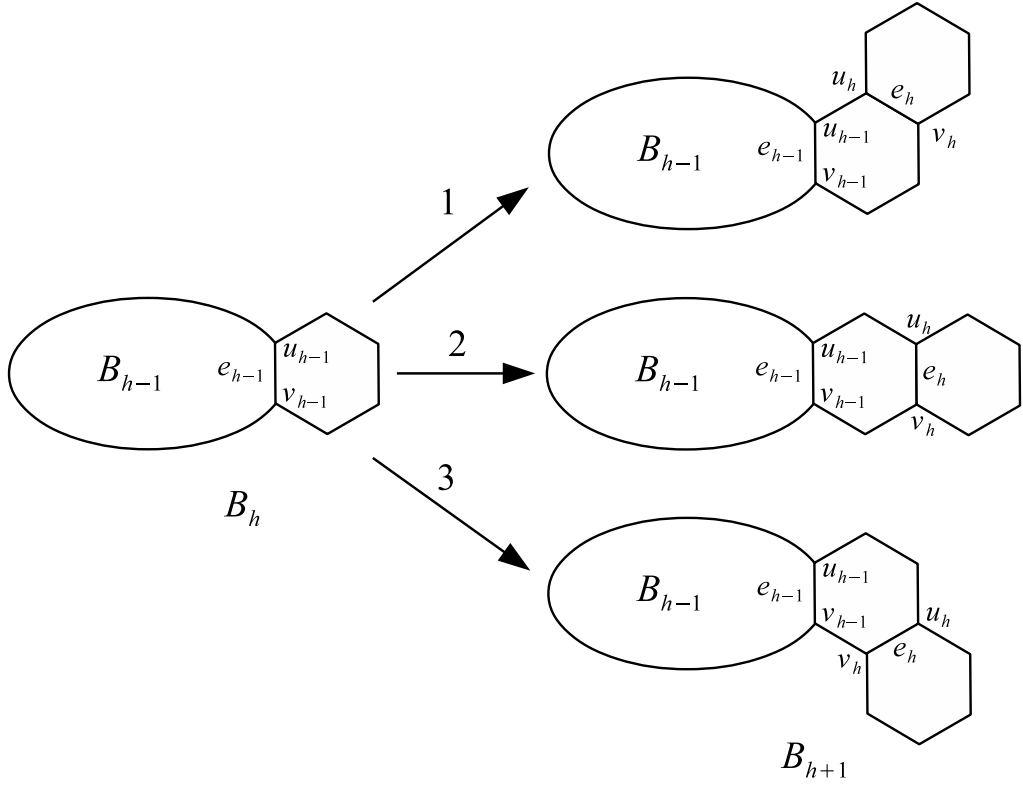}
\caption{Three possible cases of annelating a $6$-cycle to a hexagonal chain.}
	\label{moznosti}
\end{figure}

\noindent
For these three cases Theorem \ref{rek1} implies: 

\medskip\noindent
{\bf Case 1:}
\begin{eqnarray*}
H_e(B_h,u_h,x) & = & xH_e(B_{h-1},u_{h-1},x) + 2 + x + 2x^2\,, \\
H_e(B_h,v_h,x) & = & x^2H_e(B_{h-1},u_{h-1},x) + 2 + 2x + x^2\,, \\
H_e(B_h,e_h,x) & = & x^2H_e(B_{h-1},u_{h-1},x) + 1 + 2x + x^2 + x^3\,.
\end{eqnarray*}    

\noindent
{\bf Case 2:}
\begin{eqnarray*}
H_e(B_h,u_h,x) & = & x^2H_e(B_{h-1},u_{h-1},x) + 2 + 2x + x^2\,, \\
H_e(B_h,v_h,x) & = & x^2H_e(B_{h-1},v_{h-1},x) + 2 + 2x + x^2\,, \\
H_e(B_h,e_h,x) & = & x^2H_e(B_{h-1},e_{h-1},x) + 1 + 2x + x^2 + x^3\,.
\end{eqnarray*} 

\noindent
{\bf Case 3:}
\begin{eqnarray*}
H_e(B_h,u_h,x) & = & x^2H_e(B_{h-1},v_{h-1},x) + 2 + 2x + x^2\,, \\
H_e(B_h,v_h,x) & = & xH_e(B_{h-1},v_{h-1},x) + 2 + x + 2x^2\,, \\
H_e(B_h,e_h,x) & = & x^2H_e(B_{h-1},v_{h-1},x) + 1 + 2x + x^2 + x^3\,.
\end{eqnarray*} 

\noindent
We write the above recurrences in a more concise form by 
setting $\alpha_h \equiv H_e(B_h,x)$\,, $\beta_h \equiv 
H_e(B_h,u_h,x)$\,, $\gamma_h \equiv H_e(B_h,v_h,x)$\,, and $\delta_h \equiv H_e(B_h,e_h,x)$. Then we 
obtain:

\noindent
\begin{theorem} \label{gla}
Let $B_h$ be a benzenoid chain with $h$ 
hexagons. Then the edge-Hosoya polynomial $\alpha_h$ of $B_h$ satisfies 
the following recurrence 
$$
\alpha_h=\alpha_{h-1}+(x+x^2)(\beta_{h-1}+\gamma_{h-1}) + x^2\delta_{h-1}
+5+4x+3x^2+3x^3 \ ,
$$
where 
$\alpha_0 = \beta_0 = \gamma_0 = \delta_0=1$\,. Moreover, 
$\beta_h$, $\gamma_h$ and $\delta_h$ obey the following recurrences, 
depending on the cases shown in Figure \ref{moznosti}:

\noindent
{\bf Case 1:}
$$\beta_h  =  x\beta_{h-1} + 2 + x + 2x^2, \quad \gamma_h  =  x^2\beta_{h-1} + 2 + 2x + x^2, \quad \delta_h  =  x^2\beta_{h-1} + 1 + 2x + x^2 + x^3.$$

\noindent
{\bf Case 2:}
$$\beta_h  =  x^2\beta_{h-1} + 2 + 2x + x^2, \quad \gamma_h  =  x^2\gamma_{h-1} + 2 + 2x + x^2, \quad \delta_h  =  x^2\delta_{h-1} + 1 + 2x + x^2 + x^3.$$

\noindent
{\bf Case 3:}
$$\beta_h  =  x^2\gamma_{h-1} + 2 + 2x + x^2, \quad \gamma_h  =  x\gamma_{h-1} + 2 + x + 2x^2, \quad \delta_h  =  x^2\gamma_{h-1} + 1 + 2x + x^2 + x^3.$$

\end{theorem}

\section{Closed formula for polyacenes}

A hexagon $r$ of a benzenoid chain that is adjacent to two other
hexagons (that is, an {\it inner hexagon}) contains two vertices of degree 
two. We say that $r$ is {\em linearly connected} if its two vertices of
degree two are not adjacent. A benzenoid chain is called a {\it linear benzenoid chain} or a {\it polyacene} if every inner hexagon is linearly connected.

In this section we use Theorem \ref{gla} to obtain the closed formulas for the edge-Hosoya polynomial of polyacenes. Polyacene with $h$, $h \geq 1$, hexagons will be denoted by $L_h$, see Figure \ref{L4}.

\begin{figure}[!htb]
	\centering
		\includegraphics[scale=0.8, trim=0cm 0cm 0cm 0cm]{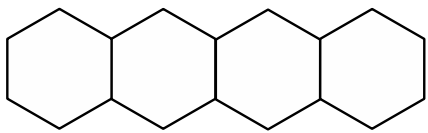}
\caption{Linear benzenoid chain (polyacene) $L_4$.}
	\label{L4}
\end{figure}

One can notice that at every step of annelation we apply Case 2 of Theorem \ref{gla}. The solutions of recursions for $\beta_h, \gamma_h$, and $\delta_h$ are:

$$\beta_h = \gamma_h = \frac{x^{2h}(2x^2+2x+1) - x^2 - 2x - 2}{x^2-1},$$

$$\delta_h = \frac{x^{2h}(x^3 + 2x^2+2x) - -x^3 - x^2 - 2x - 1}{x^2-1}.$$

\noindent
Finally, solving the recursion for $\alpha_h$ we obtain the edge-Hosoya polynomial of $L_h$:

\begin{eqnarray*}
H_e(L_h,x) & = & \frac{x^{2h+5} + 6x^{2h+4}  + 10x^{2h+3} + 6x^{2h+2} + 2x^{2h+1}}{(x^2-1)^2} \\
& + & \frac{2hx^7 - (9h+1)x^5 -(7h+5)x^4 - (h+10)x^3 + 2(h-4)x^2 + 2(4h-1)x + 5h+1}{(x^2-1)^2}.
\end{eqnarray*}

By the obtained result and Equations \eqref{ew},\eqref{ehw} it is possible to compute the closed formulas for the edge-Wiener index and the edge-hyper-Wiener index of polyacenes. However, these formulas were found in \cite{khormali,edge-hyper} and corrected in \cite{tratnik}.
\section*{Acknowledgments}

The author Petra \v Zigert Pleter\v sek acknowledge the financial support from the Slovenian Research Agency, research core funding No. P1-0297.

The author Niko Tratnik was financially supported by the Slovenian Research Agency.

\end{document}